\documentclass[reqno,final]{amsart}
\usepackage{natbib}  
\usepackage{fancyhdr} 
\usepackage{color} 
\usepackage{hyperref} 
\usepackage{graphicx} 

\usepackage{amssymb}

\usepackage{dsfont}
\usepackage{enumitem}
\usepackage[mathscr]{euscript}
\usepackage{upgreek}

\definecolor{aleacolor}{rgb}{0.16,0.59,0.78}

\hypersetup{
breaklinks,
colorlinks=true,
linkcolor=aleacolor,
urlcolor=aleacolor,
citecolor=aleacolor}

\renewcommand{\headheight}{11pt}
\renewcommand{\cite}{\citet}

\theoremstyle{plain}
\newtheorem{theorem}{Theorem}[section]                                          
\newtheorem{proposition}[theorem]{Proposition}                          
\newtheorem{lemma}[theorem]{Lemma}

\theoremstyle{definition}
\newtheorem{definition}[theorem]{Definition}
\theoremstyle{remark}
\newtheorem{remark}[theorem]{Remark}

\makeatletter \@addtoreset{equation}{section} \makeatother
\renewcommand\theequation{\thesection.\arabic{equation}}
\renewcommand\thefigure{\thesection.\arabic{figure}}
\renewcommand\thetable{\thesection.\arabic{table}}

\newcommand{\aleaIndex}[1]{\href{http://alea.impa.br/english/index_v#1.htm}{\bf #1}}
\eheader{Alea}{\aleaIndex{0}}{2018}{1}{}

\newcommand{\dd}{\mathrm{d}}
\newcommand{\ee}{\mathrm{e}}

\newcommand{\R}{\mathbb{R}}

\newcommand{\Exp}{\mathbb{E}}
\newcommand{\Var}{\mathrm{Var}}
\renewcommand{\Pr}{\mathbb{P}}
\newcommand{\ind}[1]{\mathds{1}_{\{#1\}}}

\newcommand{\petito}{\mathrm{o}}

\renewcommand{\bar}{\overline}

\renewcommand{\tilde}{\widetilde}

\newcommand{\cE}{\mathsf{E}}
\newcommand{\cD}{\mathsf{D}}

\newcommand{\CD}{\mathscr{C}(\cD)}
\newcommand{\PD}{\mathscr{P}(\cD)}
\newcommand{\PnD}{\mathscr{P}^n(\cD)}
\newcommand{\MD}{\mathscr{M}(\cD)}
\newcommand{\MnD}{\mathscr{M}^n_0(\cD)}
\newcommand{\MzD}{\mathscr{M}_0(\cD)}
\newcommand{\CzD}{\mathscr{C}_0(\cD)}

\newcommand{\genLn}{\mathbf{L}^n}
\newcommand{\genMn}{\mathbf{M}^n}
\newcommand{\genMnm}{\mathbf{M}^{n_m}}
\newcommand{\bgenM}{\bar{\mathbf{M}}}

\renewcommand{\Re}{\mathrm{Re}\,}

\begin{document}

\title{Central Limit Theorem for stationary Fleming--Viot particle systems in finite spaces}
\author{Tony Leli\`evre}
\address{Universit\'e Paris-Est, CERMICS (ENPC), INRIA\newline 77455 Marne-la-Vall\'ee, France}
\email{lelievre@cermics.enpc.fr}
\author{Loucas Pillaud-Vivien}
\address{INRIA - D\'epartement d'informatique de l'ENS, CNRS/INRIA/PSL Research University\newline 75005 Paris, France}
\email{loucas.pillaud-vivien@inria.fr}
\author{Julien Reygner}
\address{Universit\'e Paris-Est, CERMICS (ENPC)\newline 77455 Marne-la-Vall\'ee, France}
\email{julien.reygner@enpc.fr}

\thanks{This work is partially supported by the European Research Council under the European Union's Seventh Framework Programme (FP/2007-2013) / ERC Grant Agreement number 614492 and the French National Research Agency (ANR) under the program ANR-17-CE40-0030 - EFI - Entropy, flows, inequalities.}


%
%

\begin{abstract}
  We consider the Fleming--Viot particle system associated with a continuous-time Markov chain in a finite space. Assuming  irreducibility, it is known that the particle system possesses a unique stationary distribution, under which its empirical measure converges to the quasistationary distribution of the Markov chain. We complement this Law of Large Numbers with a Central Limit Theorem. Our proof essentially relies on elementary computations on the infinitesimal generator of the Fleming--Viot particle system, and involves the so-called $\pi$-return process in the expression of the asymptotic variance. Our work can be seen as an infinite-time version, in the setting of finite space Markov chains, of results by Del Moral and Miclo [ESAIM: Probab. Statist., 2003] and C\'erou, Delyon, Guyader and Rousset [arXiv:1611.00515, arXiv:1709.06771].
\end{abstract}

\maketitle

\section{Introduction}

\subsection{Quasistationary distribution} Let $\cE$ be a finite set, and $P$ be a stochastic matrix on $\cE$ with coefficients $p(x,y)$, $x,y \in \cE$. Denote by $(\mathrm{x}_t)_{t \geq 0}$ the continuous-time Markov chain in $\cE$ with infinitesimal generator
\begin{equation}\label{eq:L}
  \forall x \in \cE, \quad \forall f : \cE \to \R, \qquad Lf(x) := \sum_{y \in \cE} p(x,y)[f(y)-f(x)].
\end{equation}
For any nonempty subset $\cD$ of $\cE$, define the random time $\tau_\cD$ by
\begin{equation*}
  \tau_\cD := \inf\{t \geq 0: \mathrm{x}_t \not\in \cD\}.
\end{equation*}
A probability measure $\pi$ on $\cD$ is called a \emph{quasistationary distribution} (QSD) for $(\mathrm{x}_t)_{t \geq 0}$ in $\cD$ if for any $t \geq 0$,
\begin{equation*}
  \Pr_\pi(\mathrm{x}_t \in \cdot | \tau_\cD > t) = \pi(\cdot).
\end{equation*}
Here and throughout the paper, we use the notation $\Pr_\mu$, $\Exp_\mu$ (respectively $\Pr_x$, $\Exp_x$) to indicate that $\mathrm{x}_0$ is distributed according to the probability measure $\mu$ (respectively equal to $x$ almost surely).

As soon as the restriction of the matrix $P$ to the set $\cD$ is irreducible, there exists a unique QSD in $\cD$, which possesses a natural spectral characterisation, recalled in Proposition~\ref{prop:QSD} below. Furthermore,~\cite{DarSen67} proved that the QSD is the so-called \emph{Yaglom limit}
\begin{equation*}
  \lim_{t \to +\infty} \Pr_\mu(\mathrm{x}_t \in \cdot | \tau_\cD > t) = \pi(\cdot),
\end{equation*}
for \emph{any} initial distribution $\mu$ in the set of probability measures on $\cD$. We refer to~\cite{ColMarSan13} for a standard reference on QSDs. 

\subsection{Fleming--Viot particle system}\label{ss:intro:FV} Consider $n$ particles evolving independently in $\cE$ according to the infinitesimal generator $L$, with the supplementary condition that as soon as one particle attempts to jump to a point $y \in \cE \setminus \cD$, it is instantly moved to the location of one the $n-1$ remaining particles, picked uniformly at random. Clearly, this dynamics induces a Markov chain $(\mathrm{x}^1_t, \ldots, \mathrm{x}^n_t)_{t \geq 0}$ in $\cD^n$, which is called the \emph{Fleming--Viot particle system}\footnote{The terminology \emph{Fleming--Viot particle system} for such processes, with a finite number of particles and space-inhomogeneous exit rate, was popularised by~\cite{BurHolIngMar96,BurHolMar00} --- the second reference specifically addressing the case of $d$-dimensional Brownian motions. A discussion of the differences between the evolution of the empirical measure of this $n$-particle model and the original Fleming--Viot superprocess~\cite[Chapter~1]{Eth00} can be found in the introduction of~\cite{GriKan04}.} (\cite{AssFerGro11}).

Assume that the particles are initially iid according to some probability measure $\mu$ on $\cD$. \cite{AssFerGro11} proved that, for any $t \geq 0$, the empirical distribution $\upeta^n_t$ of $(\mathrm{x}^1_t, \ldots, \mathrm{x}^n_t)$, defined by
\begin{equation*}
  \forall x \in \cD, \qquad \upeta^n_t(x) := \frac{1}{n}\sum_{i=1}^n \ind{\mathrm{x}^i_t=x},
\end{equation*}
converges, when $n \to +\infty$, to the probability measure $\Pr_\mu(\mathrm{x}_t \in \cdot | \tau_\cD > t)$.

On the other hand, the process $(\mathrm{x}^1_t, \ldots, \mathrm{x}^n_t)_{t \geq 0}$ is irreducible in $\cD^n$, and therefore it possesses a unique stationary distribution. Let $(\mathrm{x}^1_\infty, \ldots, \mathrm{x}^n_\infty)$ be a random vector in $\cD^n$ distributed according to this stationary distribution, and consider the random probability measure $\upeta^n_\infty$ defined on $\cD$ by
\begin{equation*}
  \forall x \in \cD, \qquad \upeta^n_\infty(x) := \frac{1}{n}\sum_{i=1}^n \ind{\mathrm{x}^i_\infty=x}.
\end{equation*}
It was also proved by~\cite{AssFerGro11} that $\upeta^n_\infty$ converges, when $n \to +\infty$, to the QSD $\pi$. We refer to~\cite{BurHolMar00,DelMic03,GriKan04,Rou06,FerMar07,Lob09,Vil14,CloTha16:SPA,OcaVil17} for general approximation results for either the conditional distribution $\Pr_\mu(\mathrm{x}_t \in \cdot | \tau_\cD > t)$ or the QSD $\pi$ by Fleming--Viot-like particle systems, in various contexts. 

In particular, the convergence of $\upeta^n_\infty$ to $\pi$ provides an effective numerical procedure to approximate the QSD, by simulating the Fleming--Viot particle system over sufficiently long time intervals, and with a sufficiently large number of particles (see~\cite{AldFlaPal88,GroJon12,BenClo15,BenCloPan16} for some alternative methods). To assess the quality of this approximation, it is necessary to obtain error estimates, both on the time needed for the Fleming--Viot particle system to reach its equilibrium, and on the rate of convergence of $\upeta^n_\infty$ to $\pi$. In this article, we address the latter question by proving a Central Limit Theorem for $\upeta^n_\infty$, stated in Theorem~\ref{theo:CLT}.

\section{Notation and main results}

\subsection{Functions and measures} We denote by $\CD$ the set of \emph{functions} $\cD \to \R$ and by $\MD$ the set of finite signed \emph{measures} on $\cD$ (both sets can be identified with $\R^\cD$, but we introduce distinct notations to make our statements more clear). The duality bracket between $\MD$ and $\CD$ is written
\begin{equation*}
  \forall \rho \in \MD, \quad \forall f \in \CD, \qquad \langle \rho, f\rangle := \sum_{x \in \cD} \rho(x)f(x).
\end{equation*}

Let $\mathscr{H}$ and $\mathscr{K}$ be subspaces of $\CD$ or $\MD$, with respective dual spaces $\mathscr{H}'$ and $\mathscr{K}'$ for the bracket $\langle\cdot,\cdot\rangle$. The adjoint of an operator $R : \mathscr{H} \to \mathscr{K}$ is denoted by $R^*:\mathscr{K}' \to \mathscr{H}'$. If $\mathscr{H} \subset \CD$ and $\mathscr{K} \subset \MD$ (or alternatively $\mathscr{H} \subset \MD$ and $\mathscr{K} \subset \CD$) are each other's dual, the operator $R$ is called \emph{symmetric} if it coincides with $R^*$.

The identity of a space $\mathscr{H}$ is denoted by $I_{\mathscr{H}}$.

We denote by $\PD$ the set of probability measures on $\cD$, namely the set of measures $\eta \in \MD$ such that $\eta(x) \geq 0$ for all $x \in \cD$ and $\langle \eta, \mathbf{1}\rangle = 1$, where $\mathbf{1} \in \CD$ is defined by $\mathbf{1}(x):=1$ for all $x \in \cD$.

\subsection{Quasistationary distribution} Let us write $p_\cD(x,y) := p(x,y)$, $x,y \in \cD$, and define $P_\cD : \CD \to \CD$ by
\begin{equation*}
  \forall f \in \CD, \quad \forall x \in \cD, \qquad P_\cD f(x) := \sum_{y \in \cD} p_\cD(x,y)f(y),
\end{equation*}
and $q \in \CD$ by
\begin{equation}\label{eq:q}
  \forall x \in \cD, \qquad q(x) := \sum_{y \in \cE \setminus \cD} p(x,y) = 1 - \sum_{y \in \cD} p_\cD(x,y).
\end{equation}

Throughout the article, we work under the following assumption, which we shall not recall in the statement of our results.
\begin{enumerate}[label=(Irr),ref=Irr]
  \item\label{ass:Irr} The substochastic matrix $P_\cD$ is irreducible on the finite space $\cD$.
\end{enumerate}

The following result is a consequence of the Perron--Frobenius Theorem~(\cite[Theorem~1.5, p.~22]{Sen06}), and its proof is omitted. We recall that $(\mathrm{x}_t)_{t \geq 0}$ is the continuous-time Markov chain with infinitesimal generator $L$ introduced in~\eqref{eq:L}.
\begin{proposition}[Existence and uniqueness of the QSD]\label{prop:QSD}
  The following properties hold.
  \begin{enumerate}[label=(\roman*),ref=\roman*]
    \item\label{it:QSD:1} There exists $\lambda \in [0,1)$ such that $1-\lambda$ is the spectral radius of $P_\cD$.
    \item\label{it:QSD:5} The eigenvalue $1-\lambda$ of $P_\cD$ is simple, and any other eigenvalue $\sigma$ satisfies $\Re \sigma < 1-\lambda$.
    \item\label{it:QSD:2} There is a unique $\pi \in \PD$ such that $P_\cD^* \pi = (1-\lambda)\pi$.
    \item\label{it:QSD:2.5} For all $x \in \cD$, $\pi(x)>0$.
    \item\label{it:QSD:3} $\pi$ is the unique QSD of $(\mathrm{x}_t)_{t \geq 0}$ in $\cD$.
    \item\label{it:QSD:4} $\pi$ and $q$ satisfy the relation $\langle \pi, q\rangle = \lambda$.
  \end{enumerate}
\end{proposition}

\subsection{Fleming--Viot particle system} For all $n \geq 1$, let $\PnD$ be the subset of probability measures $\eta \in \PD$ with atoms of mass proportional to $1/n$; namely such that for all $x \in \cD$, there exists $k(x) \in \{0, \ldots, n\}$ such that $\eta(x)=k(x)/n$. Of course, $\sum_{x \in \cD}k(x)=n$. 

For all $x,y \in \cD$, let us define the measure $\theta^{x,y} \in \MD$ by 
\begin{equation}\label{eq:theta}
  \forall z \in \cD, \qquad \theta^{x,y}(z) := \ind{y=z}-\ind{x=z}.
\end{equation}

\begin{definition}[Empirical distribution of the Fleming--Viot particle system]\label{defi:FV}
  Let $n \geq 2$. The \emph{empirical distribution of the Fleming--Viot particle system} with $n$ particles is the continuous-time Markov chain $(\upeta^n_t)_{t \geq 0}$ in $\PnD$ with infinitesimal generator
  \begin{equation*}
    \genLn \phi(\eta) := \sum_{x,y \in \cD} n \eta(x)\left(p_\cD(x,y) + q(x) \frac{n\eta(y)}{n-1}\right)\left[\phi\left(\eta+\frac{\theta^{x,y}}{n}\right)-\phi(\eta)\right],
  \end{equation*}
  for any function $\phi : \PnD \to \R$.
\end{definition}
That the measure-valued process $(\upeta^n_t)_{t \geq 0}$ actually describes the evolution of the empirical distribution of the particle system introduced in Subsection~\ref{ss:intro:FV} is immediate. Indeed, in the configuration $\eta \in \PnD$, there are $n\eta(x)$ particles located at $x \in \cD$, each of which can jump `directly' to $y \in \cD$ at rate $p_\cD(x,y)$, or try to exit $\cD$ at rate $q(x)$ and be moved to $y$ with probability $n\eta(y)$ (the number of particles located at $y$) divided by $n-1$ (the total number of remaining particles).

\subsection{Statement of the main result} By Assumption~\eqref{ass:Irr}, the process $(\upeta^n_t)_{t \geq 0}$ introduced in Definition~\ref{defi:FV} is irreducible in the finite space $\PnD$, therefore it possesses a unique stationary distribution. Let $\upeta^n_\infty$ be a random variable in $\PD$ distributed according to this stationary distribution. The following Law of Large Numbers was obtained by~\cite[Theorem~2]{AssFerGro11}.

\begin{proposition}[LLN for $\upeta^n_\infty$]\label{prop:LLN}
  Let $\|\cdot\|$ be any norm on $\MD$. Then
  \begin{equation*}
    \lim_{n \to +\infty} \Exp[\|\upeta^n_\infty-\pi\|] = 0.
  \end{equation*} 
\end{proposition}

Our main result is a Central Limit Theorem complementing Proposition~\ref{prop:LLN}. For any $f \in \CD$, we denote
\begin{equation*}
  \Var_\pi(f) := \sum_{x \in \cD} \pi(x)\left(f(x)-\langle \pi, f\rangle\right)^2.
\end{equation*}
For all $t \geq 0$ and $f \in \CD$, we also define
\begin{equation}\label{eq:Qt}
  \forall x \in \cD, \qquad Q_tf(x) := \Exp_x\left[f(\mathrm{x}_t)\ind{\tau_\cD > t}\right] = \ee^{t(P_\cD-I_{\CD})}f(x).
\end{equation}

\begin{theorem}[CLT for $\upeta^n_\infty$]\label{theo:CLT}
  The \emph{fluctuation field} $\sqrt{n}(\upeta^n_\infty-\pi)$ converges in distribution, in $\MD$, to a centered Gaussian random variable with covariance operator $K : \CD \to \MD$ defined by
  \begin{equation*}
    \langle Kf, f\rangle := \Var_\pi(f) + 2\lambda \int_{s=0}^{+\infty} \ee^{2\lambda s} \Var_\pi\left(Q_s f\right) \dd s,
  \end{equation*}
  for any $f \in \CD$ such that $\langle \pi, f \rangle=0$.
\end{theorem}

In the course of the proof, we shall provide an equivalent formulation of the covariance operator $K$, which involves the semigroup of the so-called \emph{$\pi$-return process}~(\cite{GroJon12}), see Lemma~\ref{lem:idenVar}. Through a spectral argument on this semigroup, we shall in particular check that the integral in the right-hand side above is finite. 

\begin{remark}[Case $\lambda=0$]\label{rk:lambda0}
  By Proposition~\ref{prop:QSD}, $\lambda=0$ if and only if $q(x)=0$ for all $x \in \cD$, in which case the Markov chain $(\mathrm{x}_t)_{t \geq 0}$ never leaves $\cD$. In this case, $\pi$ is in fact the \emph{stationary} distribution of the chain in $\cD$. This case is formally covered by the results of this article, but it is actually trivial because the Fleming--Viot particle system merely consists in $n$ independent copies of the Markov chain $(\mathrm{x}_t)_{t \geq 0}$, to which the standard Central Limit Theorem immediately applies and yields the covariance operator $K$ characterised by $\langle Kf, f\rangle = \Var_\pi(f)$.
\end{remark}

In a much more general framework,~\cite{DelMic03}, and more recently~\cite{CerDelGuyRou16,CerDelGuyRou17}, established a Central Limit Theorem for $\upeta^n_t$, $t \in [0,+\infty)$. In the particular case of Markov chains with a finite space, and assuming for the sake of simplicity that in the Fleming--Viot particle system, the particles are initially iid according to $\pi$, the corresponding asymptotic variance for $\upeta^n_t$ reads
\begin{equation*}
  \langle K_t f,f\rangle = \Var_\pi(f) + 2\lambda \int_{s=0}^t \ee^{2\lambda s} \Var_\pi\left(Q_s f\right) \dd s,
\end{equation*}
for any $f \in \CD$ such that $\langle \pi, f\rangle=0$ (see~\cite[Proposition~3.7]{DelMic03} and~\cite[Corollary~2.2]{CerDelGuyRou16}). It is immediate to check that in the $t \to +\infty$ limit, one recovers the covariance operator $K$ from Theorem~\ref{theo:CLT}.

Owing to the remark that the Central Limit Theorem for $\eta^n_t$, $t \in [0,+\infty)$, holds for quite a large class of continuous-time Markov processes in general state spaces (see~\cite[Theorem~2.6]{CerDelGuyRou17}), it is natural to ask whether our approach can be generalised to such cases. Although our arguments can heuristically be formulated in an abstract enough setting, we expect their practical implementation to become much more technical, in particular because of the infinite-dimensional nature of $\eta^n_t$. For the sake of clarity, we therefore chose to focus on finite space Markov chains and keep the formalism elementary.  

Proposition~\ref{prop:LLN} naturally raises the question of determining a rate of convergence for $\upeta^n_\infty$ to $\pi$, and Theorem~\ref{theo:CLT} indicates that, at least asymptotically, $\sqrt{n}$ is the correct answer. At the nonasymptotic level, an inequality of the form
\begin{equation}\label{eq:varestim}
  \forall n \geq 2, \qquad \Exp\left[\|\upeta^n_\infty-\pi\|^2\right] \leq \frac{C}{n},
\end{equation}
for some norm $\|\cdot\|$ on $\MD$, may therefore be expected to hold. In the terms of~\cite{FerMar07} and~\cite{AssFerGro11}, such a control essentially amounts to asserting that, under the stationary distribution of the Fleming--Viot particle system, `the correlations between the particles are of order $1/n$'. At finite time $t$, this estimate was proved by~\cite[Proposition~2]{AssFerGro11} and is an important ingredient of the proof of Proposition~\ref{prop:LLN}. However, to the best of our knowledge, the control of the correlations under the stationary distribution is known to hold only under specific mixing conditions~(\cite{FerMar07,CloTha16:SPA}). It also appears in~\cite{Rou06} for the case of diffusions with soft killing. Our proof relies on a weaker form of~\eqref{eq:varestim}, further comments on this aspect are provided in Remark~\ref{rk:var}.

\subsection{Sketch of the proof and outline of the article} In Section~\ref{s:prelim}, we introduce the $\pi$-return process, which describes the trajectorial behaviour of the stationary Fleming--Viot particle system in the $n \to +\infty$ limit, and with which several of the notions that we shall manipulate are related.

The proof of Theorem~\ref{theo:CLT} follows the standard approach of: (i) proving the tightness of the fluctuation field $\sqrt{n}(\upeta^n_\infty-\pi)$; (ii) identifying the law of its limits. The first step is detailed in Section~\ref{s:tight}. It rests on a moment estimate derived from algebraic manipulations on the infinitesimal generator $\genLn$, and the use of an appropriate Lyapunov equation. The second step is detailed in Section~\ref{s:gauss}, where the law of any limit of $\sqrt{n}(\upeta^n_\infty-\pi)$ is shown to be the stationary distribution of a linear diffusion process, for which uniqueness and identification follow from standard arguments. 

Elementary linear algebra results, which are suited to our framework, are collected in Appendix~\ref{app:linalg}.

Apart from the use of Proposition~\ref{prop:LLN} made in the proof of our tightness result, we emphasise that our arguments are entirely static, in the sense that they merely involve estimates on the law of $\upeta^n_\infty$, which stem from elementary manipulations of the infinitesimal generator $\genLn$. At the technical level, we thereby avoid resorting to graphical constructions of the process~(\cite{AssFerGro11,GroJon12}), coupling techniques~(\cite{CloTha16:SPA,CloTha16:ALEA}) or martingale arguments~(\cite{CerDelGuyRou16,CerDelGuyRou17}).

Throughout the article, we take the convention to call `Theorem' and `Lemma' the results which are proper to our arguments, while we call `Proposition' the results which are either proved elsewhere or essentially standard.

\section{The \texorpdfstring{$\pi$}{pi}-return process}\label{s:prelim} 

\subsection{Definition}\label{ss:defipireturn} By exchangeability, the convergence of the empirical distribution $\upeta^n_\infty$ to $\pi$ stated in Proposition~\ref{prop:LLN} is known to be equivalently formulated as a \emph{chaoticity} result~(\cite{Szn91}); namely, under the stationary distribution of the Fleming--Viot particle system, any finite subset of particles asymptotically behaves, when $n \to +\infty$, like independent realisations of $\pi$. \cite{GroJon12} provided a trajectorial description of this chaoticity phenomenon (see also~\cite{GriKan06,Lob09} for related results), based on the following process. 
\begin{definition}[$\pi$-return process]
  The $\pi$-return process is the continuous-time Markov chain $(\mathrm{x}^\pi_t)_{t \geq 0}$ in $\cD$ with jump rates
  \begin{equation*}
    \forall x,y \in \cD, \qquad p_\cD^\pi(x,y) := p_\cD(x,y) + q(x) \pi(y).
  \end{equation*}
\end{definition}
The process $(\mathrm{x}^\pi_t)_{t \geq 0}$ can be simply described as a copy of $(\mathrm{x}_t)_{t \geq 0}$ in which the transitions to points in $\cE \setminus \cD$ are replaced with transitions to points in $\cD$ independently drawn according to $\pi$. By Assumption~\eqref{ass:Irr} and Proposition~\ref{prop:QSD}, the QSD $\pi$ is the unique stationary distribution of $(\mathrm{x}^\pi_t)_{t \geq 0}$. In~\cite[Theorem~2.10]{GroJon12}, it was proved that if the Fleming--Viot particle system is started under its stationary distribution, then when $n \to +\infty$, the marginal evolution of each particle converges in distribution, in the space of sample-paths, to the process $(\mathrm{x}^\pi_t)_{t \geq 0}$ with initial distribution $\pi$.

We denote by $P^\pi_\cD : \CD \to \CD$ the operator defined by
\begin{equation*}
  \forall x \in \cD, \quad \forall f \in \CD, \qquad P^\pi_\cD f(x) := \sum_{y \in \cD} p_\cD^\pi(x,y)f(y),
\end{equation*}
so that
\begin{equation}\label{eq:PpiPD}
  P^\pi_\cD f = P_\cD f + \langle \pi, f\rangle q.
\end{equation}
The infinitesimal generator $L^\pi_\cD : \CD \to \CD$ of the $\pi$-return process writes $L^\pi_\cD = P^\pi_\cD - I_{\CD}$, and the associated semigroup is denoted by
\begin{equation}\label{eq:PtD}
  \forall t \geq 0, \quad \forall x \in \cD, \quad \forall f \in \CD, \qquad P_{t, \cD}^\pi f(x) := \Exp_x[f(\mathrm{x}^\pi_t)] = \ee^{t L^\pi_\cD}f(x).
\end{equation}

Finally, the Dirichlet form of the $\pi$-return process is the quadratic form $\mathcal{A}^\pi_\cD$ on $\CD$ defined by
\begin{equation*}
  \mathcal{A}^\pi_\cD(f) := \frac{1}{2} \sum_{x,y \in \cD} \pi(x)p^\pi_\cD(x,y)\left[f(y)-f(x)\right]^2 = -\sum_{x \in \cD} \pi(x) f(x) L^\pi_\cD f(x).
\end{equation*}

\subsection{Spectral estimates} In this subsection, we describe some spectral properties of the $\pi$-return process. We first introduce the spectral gap of the operator $P_\cD$, which by Proposition~\ref{prop:QSD} is positive.

\begin{definition}[Spectral gap of $P_\cD$]\label{defi:gamma}
  The \emph{spectral gap} of $P_\cD$ is defined by
  \begin{equation*}
    \gamma := 1-\lambda - \max_{\sigma \not= 1-\lambda} \Re \sigma > 0,
  \end{equation*}
  where the max is taken over the eigenvalues $\sigma$ of $P_\cD$.
\end{definition}

Since $\pi$ is the unique stationary distribution of the $\pi$-return process, $1$ is a single eigenvalue for $P^\pi_\cD$, and
\begin{equation}\label{eq:Ppi1}
  P^\pi_\cD\mathbf{1}=\mathbf{1}, \qquad (P^\pi_\cD)^*\pi=\pi.
\end{equation}
As a consequence, we have the direct sum decomposition
\begin{equation}\label{eq:CzD}
  \CD = \R\mathbf{1} \oplus \CzD, \qquad \CzD := \{f \in \CD: \langle\pi,f\rangle=0\},
\end{equation}
where both subspaces $\R\mathbf{1}$ and $\CzD$ are stable by $P^\pi_\cD$. Likewise, 
\begin{equation}\label{eq:MzD}
  \MD = \R\pi \oplus \MzD, \qquad \MzD := \{\xi \in \MD: \langle\xi,\mathbf{1}\rangle=0\},
\end{equation}
and both subspaces $\R\pi$ and $\MzD$ are stable by $(P^\pi_\cD)^*$.

\begin{remark}[Dual spaces of $\MzD$ and $\CzD$]\label{rk:idendual}
  Throughout the article, we identify the spaces $\MzD$ and $\CzD$ as each other's dual.
\end{remark}

We may now state some useful properties of $P^\pi_\cD$. We recall the definition~\eqref{eq:Qt} of the semigroup of operators $(Q_t)_{t \geq 0}$.

\begin{lemma}[Spectral properties of the $\pi$-return process]\label{lem:spectpi}
  The operator $P^\pi_\cD$ possesses the following properties.
  \begin{enumerate}[label=(\roman*),ref=\roman*]
    \item\label{it:spectpi:1} Any complex eigenvalue $\sigma\not=1$ of $P^\pi_\cD$ satisfies $\Re \sigma \leq 1-\lambda-\gamma$.
    \item\label{it:spectpi:2} For any $\delta>0$, for any norm $\|\cdot\|$ on $\CD$, there exists $C_\delta \in [0,+\infty)$ such that, for any $f \in \CzD$,
    \begin{equation*}
      \forall t \geq 0, \qquad \max_{x \in \cD} P^\pi_{t,\cD}f(x) \leq C_\delta \ee^{-t(\lambda+\gamma-\delta)}\|f\|.
    \end{equation*}
    \item\label{it:spectpi:3} For any $f \in \CD$, $x \in \cD$ and $t \geq 0$, $P^\pi_{t,\cD}f(x) = Q_t(f-\langle \pi,f\rangle\mathbf{1})(x) + \langle \pi,f\rangle$.
  \end{enumerate}
\end{lemma}
\begin{proof}
  All three assertions of the lemma follow from the observation that, by~\eqref{eq:CzD} and~\eqref{eq:PpiPD}, the restriction of $P^\pi_\cD$ to the stable subspace $\CzD$ coincides with $P_\cD$. Then any eigenvalue of $P^\pi_\cD$ which is not equal to $1$ is necessarily an eigenvalue of $P_\cD$ and thus satisfies Assertion~\eqref{it:spectpi:1} as a consequence of Definition~\ref{defi:gamma}. The latter assertion implies that, for any $\delta>0$, the family of operators $\{\ee^{t(\lambda+\gamma-\delta)}P^\pi_{t,\cD}, t \geq 0\}$ is bounded on $\CzD$ and thereby yields Assertion~\eqref{it:spectpi:2}. One finally obtains Assertion~\eqref{it:spectpi:3} by noting that the operators $Q_t = \ee^{t (P_\cD-I_{\CD})}$ and $P^\pi_{t,\cD} = \ee^{t(P^\pi_\cD-I_{\CD})}$ coincide on $\CzD$.
\end{proof}

\subsection{Two related operators} To prove Theorem~\ref{theo:CLT}, we shall show that $\sqrt{n}(\upeta^n_\infty-\pi)$ converges to the stationary distribution of a linear diffusion process in $\MzD$, with drift and diffusion operators expressed in terms of quantities related to the $\pi$-return process. We introduce these operators in the present subsection and will use them in Sections~\ref{s:tight} and~\ref{s:gauss}.

\subsubsection{The drift operator} We recall that by~\eqref{eq:Ppi1} and~\eqref{eq:MzD}, the subspace $\MzD$ of $\MD$ is stable by the adjoint $(P^\pi_\cD)^* : \MD \to \MD$ of the operator $P^\pi_\cD : \CD \to \CD$.

\begin{definition}[Drift operator]\label{defi:B0}
  The \emph{drift operator} is the operator $B_0 : \MzD \to \MzD$ defined by
  \begin{equation*}
    B_0 := (P^\pi_\cD)^* - (1-\lambda) I_{\MzD} = (L^\pi_\cD)^* + \lambda I_{\MzD}.
  \end{equation*}
\end{definition}

\begin{remark}[Spectrum]\label{rk:spB0}
  By Lemma~\ref{lem:spectpi}, any eigenvalue $\tau$ of $B_0$ satisfies $\Re \tau \leq -\gamma$.
\end{remark}

\subsubsection{The diffusion operator} Recall the definition of the Dirichlet form $\mathcal{A}^\pi_\cD$ of the $\pi$-return process in Subsection~\ref{ss:defipireturn}. Since, by Remark~\ref{rk:idendual}, the dual of $\CzD$ is identified with $\MzD$, by Riesz' Theorem, there exists a symmetric operator $A^\pi_\cD : \CzD \to \MzD$ such that, for any $f \in \CzD$,
\begin{equation*}
  \langle A^\pi_\cD f, f\rangle = \mathcal{A}^\pi_\cD(f).
\end{equation*}
We call $A^\pi_\cD$ the \emph{diffusion operator}; notice that it can be expressed as the restriction to $\CzD$ of the symmetric part of $-L^\pi_\cD$ in $L^2(\pi)$.

The irreducibility of the $\pi$-return process implies the following result (elementary material on symmetric and positive definite operators is gathered in Appendix~\ref{app:linalg}).
\begin{lemma}[On the operator $A^\pi_\cD$]\label{lem:ApiD}
  The operator $A^\pi_\cD$ is positive definite, in the sense that it satisfies $\langle A^\pi_\cD f, f\rangle > 0$, for all $f \in \CzD\setminus\{0\}$.
\end{lemma}

\section{Tightness of the fluctuation field}\label{s:tight}

In this section, we prove the following result.

\begin{lemma}[Tightness of the fluctuation field]\label{lem:tight}
  Let $\|\cdot\|$ be any norm on $\MD$. For all $\epsilon > 0$, there exists $r_\epsilon \in (0,+\infty)$ such that, for all $n \geq 2$,
  \begin{equation*}
    \Pr(\sqrt{n}\|\upeta^n_\infty-\pi\| > r_\epsilon) \leq \epsilon.
  \end{equation*}
\end{lemma}

An auxiliary moment estimate is established in Subsection~\ref{ss:tight:mom}. The proof of Lemma~\ref{lem:tight} is detailed in Subsection~\ref{ss:tight:pf}.
\subsection{Moment estimate}\label{ss:tight:mom} We first state a moment estimate for $\upeta^n_\infty$. We recall the Definition~\ref{defi:B0} of the drift operator $B_0 : \MzD \to \MzD$, and the definition~\eqref{eq:q} of $q \in \CD$.

\begin{lemma}[Moment estimate]\label{lem:somom}
  Let $R : \MzD \to \CzD$ be a symmetric operator. There exists $C(R) \in [0,+\infty)$ such that, for all $n \geq 2$,
  \begin{equation*}
    \left|\Exp\left[\left\langle B_0(\upeta^n_\infty-\pi) + \langle \upeta^n_\infty-\pi, q\rangle (\upeta^n_\infty-\pi), R(\upeta^n_\infty-\pi)\right\rangle\right]\right| \leq \frac{C(R)}{n}.
  \end{equation*}
\end{lemma}
\begin{proof}
  In the proof, we shall employ the operator $Q : \MD \to \MD$ defined by
  \begin{equation*}
    \forall \rho \in \MD, \quad \forall x \in \cD, \qquad Q\rho(x) := q(x)\rho(x).
  \end{equation*}
  
  Let $R : \MzD \to \CzD$ be a symmetric operator. Let us define the function $\phi : \PD \to \R$ by
  \begin{equation*}
    \forall \eta \in \PD, \qquad \phi(\eta) := \frac{1}{2}\langle \eta-\pi, R(\eta-\pi)\rangle.
  \end{equation*}
  For all $x,y \in \cD$, for all $\eta \in \PD$,
  \begin{equation*}
    \phi\left(\eta + \frac{\theta^{x,y}}{n}\right) - \phi(\eta) =  \frac{1}{n}\langle \theta^{x,y}, R(\eta-\pi)\rangle + \frac{1}{2n^2} \langle \theta^{x,y}, R\theta^{x,y}\rangle,
  \end{equation*}
  where we recall the definition~\eqref{eq:theta} of $\theta^{x,y}$. 
  
  Since the definition of $\upeta^n_\infty$ implies that $\Exp[\genLn\phi(\upeta^n_\infty)]=0$, we get the identity
  \begin{equation}\label{eq:pf:estim:1}
    \begin{aligned}
      0 &= \Exp\left[\sum_{x,y \in \cD} \upeta^n_\infty(x)\left(p_\cD(x,y)+q(x)\frac{n\upeta^n_\infty(y)}{n-1}\right)\langle \theta^{x,y}, R(\upeta^n_\infty-\pi)\rangle\right]\\
      & \quad + \frac{1}{2n}\Exp\left[\sum_{x,y \in \cD} \upeta^n_\infty(x)\left(p_\cD(x,y)+q(x)\frac{n\upeta^n_\infty(y)}{n-1}\right) \langle \theta^{x,y}, R\theta^{x,y}\rangle\right].
    \end{aligned}
  \end{equation}
  We first make explicit the value of the first expectation in the right-hand side of~\eqref{eq:pf:estim:1}. By~\eqref{eq:theta} and~\eqref{eq:q}, for all $\eta \in \PD$, for all $z \in \cD$,
  \begin{equation*}\label{eq:pf:estim:2}
    \begin{aligned}
      & \sum_{x,y \in \cD} \eta(x)\left(p_\cD(x,y)+q(x)\frac{n\eta(y)}{n-1}\right)\theta^{x,y}(z)\\
      & \qquad = \sum_{x \in \cD} \eta(x)\left(p_\cD(x,z)+q(x)\frac{n\eta(z)}{n-1}\right) - \sum_{y \in \cD} \eta(z)\left(p_\cD(z,y)+q(z)\frac{n\eta(y)}{n-1}\right)\\
      & \qquad = P_\cD^*\eta(z) + \frac{n}{n-1}\langle \eta, q\rangle \eta(z) - \eta(z)(1-q(z)) - \frac{n}{n-1}\eta(z)q(z)\\
      & \qquad = \left(P_\cD^*-I_{\MD}-\frac{Q}{n-1}\right)\eta(z) + \frac{n}{n-1}\langle \eta, q\rangle \eta(z)\\
      & \qquad = \left(P_\cD^*-I_{\MD}\right)\eta(z) + \langle \eta, q\rangle \eta(z) + \frac{1}{n-1}\left\{-Q\eta(z) + \langle \eta, q\rangle \eta(z)\right\},
    \end{aligned}
  \end{equation*}
  so that 
  \begin{equation}\label{eq:pf:estim:3}
    \begin{aligned}
      & \Exp\left[\sum_{x,y \in \cD} \upeta^n_\infty(x)\left(p_\cD(x,y)+q(x)\frac{n\upeta^n_\infty(y)}{n-1}\right)\langle \theta^{x,y}, R(\upeta^n_\infty-\pi)\rangle\right]\\
      & \qquad= \Exp\left[\left\langle \left(P_\cD^*-I_{\MD}\right)\upeta^n_\infty + \langle \upeta^n_\infty, q\rangle \upeta^n_\infty, R(\upeta^n_\infty-\pi)\right\rangle\right]\\
      & \qquad\quad + \frac{1}{n-1}\left\{\Exp\left[\left\langle -Q\upeta^n_\infty + \langle \upeta^n_\infty, q\rangle \upeta^n_\infty, R(\upeta^n_\infty-\pi)\right\rangle\right]\right\}.
    \end{aligned}
  \end{equation}
  
  Using Proposition~\ref{prop:QSD}, we now rewrite, for all $\eta \in \PD$,
  \begin{equation*}
    (P_\cD^*-I_{\MD})\eta + \langle \eta, q\rangle \eta = (P_\cD^*-(1-\lambda)I_{\MD})(\eta-\pi) + \langle \eta-\pi, q\rangle (\eta-\pi) + \langle \eta-\pi, q\rangle \pi, 
  \end{equation*}
  so that
  \begin{equation*}
    \left\langle \left(P_\cD^*-I_{\MD}\right)\eta + \langle \eta, q\rangle \eta, R(\eta-\pi)\right\rangle= \left\langle B_0(\eta-\pi) + \langle \eta-\pi, q\rangle (\eta-\pi), R(\eta-\pi)\right\rangle,
  \end{equation*}
  where we have used the Definition~\ref{defi:B0} of $B_0$ as well as the fact that $R(\eta-\pi)\in\CzD$ to justify that $\langle \langle \eta-\pi, q\rangle \pi, R(\eta-\pi)\rangle$ vanishes. As a consequence, the first expectation in the right-hand side of~\eqref{eq:pf:estim:3} rewrites
  \begin{equation}\label{eq:pf:estim:4}
    \begin{aligned}
      & \Exp\left[\left\langle \left(P_\cD^*-I_{\MD}\right)\upeta^n_\infty + \langle \upeta^n_\infty, q\rangle \upeta^n_\infty, R(\upeta^n_\infty-\pi)\right\rangle\right]\\
      & \qquad = \Exp\left[\left\langle B_0(\upeta^n_\infty-\pi) + \langle \upeta^n_\infty-\pi, q\rangle (\upeta^n_\infty-\pi), R(\upeta^n_\infty-\pi)\right\rangle\right].
    \end{aligned}
  \end{equation}
  
  Combining~\eqref{eq:pf:estim:1}, \eqref{eq:pf:estim:3} and~\eqref{eq:pf:estim:4}, we finally get the identity
  \begin{equation*}
    \begin{aligned}
      & \Exp\left[\left\langle B_0(\upeta^n_\infty-\pi) + \langle \upeta^n_\infty-\pi, q\rangle (\upeta^n_\infty-\pi), R(\upeta^n_\infty-\pi)\right\rangle\right]\\
      & \qquad= - \frac{1}{n-1}\Exp\left[\left\langle -Q\upeta^n_\infty + \langle \upeta^n_\infty, q\rangle \upeta^n_\infty, R(\upeta^n_\infty-\pi)\right\rangle\right]\\
      & \qquad\quad - \frac{1}{2n}\Exp\left[\sum_{x,y \in \cD} \upeta^n_\infty(x)\left(p_\cD(x,y)+q(x)\frac{n\upeta^n_\infty(y)}{n-1}\right)\langle \theta^{x,y}, R\theta^{x,y}\rangle\right],
    \end{aligned}
  \end{equation*}
  the right-hand side of which is bounded in modulus by $C(R)/n$ for some constant $C(R) \in [0,+\infty)$ depending on $R$.
\end{proof}

\subsection{Proof of Lemma~\ref{lem:tight}}\label{ss:tight:pf} We are now ready to present the proof of Lemma~\ref{lem:tight}. 

\begin{proof}[Proof of Lemma~\ref{lem:tight}] Let us fix a norm $\|\cdot\|$ on $\MD$, and $\epsilon > 0$. The proof is divided into 3 steps.

  \emph{Step~1.} Recall the Definition~\ref{defi:gamma} of the spectral gap $\gamma>0$ of $P_\cD$. By Proposition~\ref{prop:LLN}, there exists $n_0 \geq 2$ depending on $\epsilon$ such that, for any $n \geq n_0$,
  \begin{equation*}
    \Pr(|\langle\upeta^n_\infty-\pi,q\rangle| \geq \gamma/2) \leq \epsilon/2.
  \end{equation*}
  As a consequence, using Markov's inequality, we get that for any $r \in (0,+\infty)$, for any $n \geq n_0$,
  \begin{equation*}
    \Pr(\sqrt{n}\|\upeta^n_\infty-\pi\| \geq r) \leq \frac{1}{r^2}\Exp\left[n\|\upeta^n_\infty-\pi\|^2\ind{|\langle\upeta^n_\infty-\pi,q\rangle| < \gamma/2}\right] + \frac{\epsilon}{2}.
  \end{equation*}
  In the next step, we use Lemma~\ref{lem:somom} and Proposition~\ref{prop:Lyapgen} to control the expectation in the right-hand side.
  
  \emph{Step~2.} Let us rewrite the result of Lemma~\ref{lem:somom} as
  \begin{equation}\label{eq:pfmoment:0}
    \left|\Exp\left[\left\langle \left(B_0+\langle \upeta^n_\infty-\pi, q\rangle I_{\MzD}\right)(\upeta^n_\infty-\pi), R(\upeta^n_\infty-\pi)\right\rangle\right]\right| \leq \frac{C(R)}{n},
  \end{equation}
  for any symmetric operator $R : \MzD \to \CzD$. Now, for any $\eta \in \PD$,
  \begin{equation}\label{eq:pfmoment:1}
    B_0+\langle \eta-\pi,q\rangle I_{\MzD} = \tilde{B}_0 -\left(\frac{\gamma}{2}-\langle \eta-\pi,q\rangle\right) I_{\MzD},
  \end{equation}
  with $\tilde{B}_0 := B_0 + (\gamma/2) I_{\MzD}$. By Remark~\ref{rk:spB0}, all eigenvalues of $\tilde{B}_0$ have a negative real part, therefore applying Proposition~\ref{prop:Lyapgen} for any choice of a symmetric and positive definite operator $A : \CzD \to \MzD$, we obtain that there exists a symmetric and positive definite operator $\tilde{K} : \CzD \to \MzD$ such that
  \begin{equation*}
    \tilde{B}_0 \tilde{K} + \tilde{K} \tilde{B}_0^* + 2A = 0.
  \end{equation*}
  By Proposition~\ref{prop:sympos}, $\tilde{K}$ is invertible and $R := \tilde{K}^{-1}$ is symmetric and positive definite. We thus get, for any $\eta \in \PD$,
  \begin{equation*}
    \langle \tilde{B}_0(\eta-\pi), R(\eta-\pi)\rangle = -\langle A R(\eta-\pi), R(\eta-\pi)\rangle.
  \end{equation*}
  Let $\|\cdot\|_{\CzD}$ be a norm on $\CzD$, and let $\|\cdot\|_{\MzD}$ be the norm induced by $\|\cdot\|$ on $\MzD$. Since $A$ is symmetric and positive definite, by Proposition~\ref{prop:sympos} there exists $c_A \in (0,+\infty)$ such that for all $f \in \CzD$, $\langle Af, f\rangle \geq c_A \|f\|_{\CzD}^2$. As a consequence,
  \begin{equation*}
    \langle A R(\eta-\pi), R(\eta-\pi)\rangle \geq c_A \|R(\eta-\pi)\|_{\CzD}^2 \geq \frac{c_A}{|||\tilde{K}|||^2} \|\eta-\pi\|_{\MzD}^2,
  \end{equation*}
  where we denote by $|||\cdot|||$ the operator norm between the spaces $(\CzD,\|\cdot\|_{\CzD})$ and $(\MzD,\|\cdot\|_{\MzD})$. We deduce that
  \begin{equation}\label{eq:pfmoment:2}
    \langle \tilde{B}_0(\eta-\pi), R(\eta-\pi)\rangle \leq -\frac{c_A}{|||\tilde{K}|||^2} \|\eta-\pi\|_{\MzD}^2.
  \end{equation}
  On the other hand, if $\eta \in \PD$ is such that $|\langle\eta-\pi,q\rangle| < \gamma/2$, then
  \begin{equation}\label{eq:pfmoment:3}
    \left(\frac{\gamma}{2}-\langle \eta-\pi,q\rangle\right)\langle \eta-\pi, R(\eta-\pi)\rangle \geq 0,
  \end{equation}
  where we have used the fact that $R$ is symmetric and positive definite. Combining~\eqref{eq:pfmoment:1}, \eqref{eq:pfmoment:2} and \eqref{eq:pfmoment:3}, we deduce that, for all $\eta \in \PD$,
  \begin{equation*}
    \|\eta-\pi\|^2 \ind{|\langle \eta-\pi,q\rangle| < \gamma/2} \leq -\frac{|||\tilde{K}|||^2}{c_A} \left\langle \left(B_0+\langle \eta-\pi, q\rangle I_{\MzD}\right)(\eta-\pi), R(\eta-\pi)\right\rangle.
  \end{equation*}
  Evaluating this inequality for $\eta=\upeta^n_\infty$, taking the expectation and applying~\eqref{eq:pfmoment:0}, we get
  \begin{equation*}
    \Exp\left[n\|\upeta^n_\infty-\pi\|^2 \ind{|\langle \upeta^n_\infty-\pi,q\rangle| < \gamma/2}\right] \leq  C' := \frac{|||\tilde{K}|||^2}{c_A}C(\tilde{K}^{-1}).
  \end{equation*}
  
  \emph{Step~3.} The final estimates of Steps~1 and~2 show that any choice of $r \geq \sqrt{2C'/\epsilon}$ implies $\Pr(\sqrt{n}\|\upeta^n_\infty-\pi\| \geq r) \leq \epsilon$ for all $n \geq n_0$. On the other hand, it is known that any finite family of random variables in $\MD$ is tight~(\cite[Theorem~1.3, p.~8]{Bil99}), therefore there exists $r'_\epsilon \in (0,+\infty)$ such that $\Pr(\sqrt{n}\|\upeta^n_\infty-\pi\| \geq r'_\epsilon) \leq \epsilon$, for all $n < n_0$. Taking $r_\epsilon$ as the maximum between $r'_\epsilon$ and $\sqrt{2C'/\epsilon}$ completes the proof.
\end{proof}

\begin{remark}[Variance estimates]\label{rk:var}
  Assume that $q$ and $\gamma$ satisfy the condition
  \begin{equation}\label{eq:gamma-qq}
    \alpha := \max_{x \in \cD} q(x) - \min_{x \in \cD} q(x) < \gamma.
  \end{equation}
  Then employing the decomposition
  \begin{equation*}
    B_0 + \langle \eta-\pi,q\rangle I_{\MzD} = \tilde{B}_0 - \left(\alpha - \langle \eta-\pi,q\rangle\right) I_{\MzD}, \quad \tilde{B}_0 := B_0 + \alpha I_{\MzD},
  \end{equation*}
  in Step~2 of the proof of Lemma~\ref{lem:tight} leads to the variance estimate~\eqref{eq:varestim}, which holds without any smallness condition on $\|\upeta^n_\infty-\pi\|$. Such an estimate was also obtained by~\cite[Theorem~1.3]{CloTha16:SPA}, under an assumption similar to~\eqref{eq:gamma-qq}. The latter assumption is in particular satisfied if $q(x)=\lambda$ for any $x \in \cD$, that is to say, the rate at which the Markov chain $(\mathrm{x}_t)_{t \geq 0}$ exits $\cD$ does not depend on its current position. An example of such a chain is studied by the same authors in~\cite[Section~3.1]{CloTha16:SPA} and~\cite[Section~2]{CloTha16:ALEA}.
\end{remark}

\section{Asymptotic normality of the fluctuation field}\label{s:gauss}

In this section, we complete the proof of Theorem~\ref{theo:CLT} by identifying the law of any limit of $\sqrt{n}(\upeta^n_\infty-\pi)$. Let us sketch our argument. In Subsection~\ref{ss:gauss:fluc}, we interpret the law of $\sqrt{n}(\upeta^n_\infty-\pi)$ as the stationary distribution of a continuous-time Markov chain $(\upxi^n_t)_{t \geq 0}$ in $\MzD$. In Subsection~\ref{ss:gauss:conv}, we describe the $n \to +\infty$ limit of the infinitesimal generator of $(\upxi^n_t)_{t \geq 0}$. In Subsection~\ref{ss:gauss:bgenM}, we show that this limit is the infinitesimal generator of a linear diffusion process in $\MzD$, the unique stationary distribution of which is the Gaussian measure introduced in Theorem~\ref{theo:CLT}. The proof of the latter theorem is completed in Subsection~\ref{ss:gauss:pf}.

\subsection{The process \texorpdfstring{$(\upxi^n_t)_{t \geq 0}$}{xi}}\label{ss:gauss:fluc}
Let $\Phi^n : \PnD \to \MzD$ be defined by
\begin{equation*}
  \forall \eta \in \PnD, \qquad \Phi^n(\eta) := \sqrt{n}(\eta-\pi),
\end{equation*}
and let $\MnD \subset \MzD$ denote the range of $\Phi^n$. For all $t \geq 0$, we define
\begin{equation*}
  \upxi^n_t := \Phi^n(\upeta^n_t) = \sqrt{n}\left(\upeta^n_t-\pi\right).
\end{equation*}
Since $\Phi^n$ is a one-to-one map between $\PnD$ and $\MnD$, $(\upxi^n_t)_{t \geq 0}$ is a continuous-time Markov chain, with infinitesimal generator
\begin{equation*}
  \genMn \psi(\xi) := \genLn (\psi\circ\Phi^n)((\Phi^n)^{-1}(\xi))= \sum_{x,y \in \cD} \vartheta_n(x,y,\xi)\left[\psi\left(\xi+\frac{\theta^{x,y}}{\sqrt{n}}\right)-\psi(\xi)\right],
\end{equation*}
where
\begin{equation*}
  \vartheta_n(x,y,\xi) := n \left(\pi(x) + \frac{\xi(x)}{\sqrt{n}}\right)\left(p_\cD(x,y) + q(x) \frac{n}{n-1}\left(\pi(y) + \frac{\xi(y)}{\sqrt{n}}\right)\right).
\end{equation*}
Besides, the law of the random variable 
\begin{equation}\label{eq:upxininf}
  \upxi^n_\infty := \Phi^n(\upeta^n_\infty) = \sqrt{n}\left(\upeta^n_\infty-\pi\right)
\end{equation}
is the unique stationary distribution of $(\upxi^n_t)_{t \geq 0}$, so that Theorem~\ref{theo:CLT} reduces to proving the convergence in distribution, in $\MzD$, of $\upxi^n_\infty$.

\subsection{Convergence of \texorpdfstring{$\genMn$}{Mn}}\label{ss:gauss:conv} For any smooth function $\psi : \MzD \to \R$, the gradient $\nabla \psi(\xi) \in \CzD$ and the Hessian matrix $\nabla^2 \psi(\xi) : \MzD \to \CzD$ are defined by the Taylor expansion
\begin{equation*}
  \forall \xi, \zeta \in \MzD, \qquad \psi(\xi+\epsilon\zeta) = \psi(\xi) + \epsilon \langle \zeta, \nabla\psi(\xi)\rangle + \frac{\epsilon^2}{2} \langle \zeta, \nabla^2\psi(\xi)\zeta\rangle + \petito(\epsilon^2).
\end{equation*}

Besides, for any symmetric operator $R : \MzD \to \CzD$, we introduce the notation
\begin{equation}\label{eq:A::}
  A^\pi_\cD::R := \frac{1}{2} \sum_{x,y \in \cD} \pi(x)p^\pi_\cD(x,y)\langle \theta^{x,y}, R\theta^{x,y}\rangle,
\end{equation}
where we recall the definition~\eqref{eq:theta} of $\theta^{x,y} \in \MzD$.

\begin{lemma}[Convergence of $\genMn$]\label{lem:cvMn}
  Let $\psi : \MzD \to \R$ be a $C^\infty$ function with compact support. We have
  \begin{equation*}
    \lim_{n \to +\infty} \sup_{\xi \in \MzD} \left|\genMn \psi(\xi)-\bgenM\psi(\xi)\right| = 0,
  \end{equation*}
  where
  \begin{equation*}
    \bgenM \psi(\xi) := \langle B_0 \xi, \nabla\psi(\xi)\rangle + A^\pi_\cD :: \nabla^2\psi(\xi).
  \end{equation*}
\end{lemma}
\begin{proof}
  Let $\psi : \MzD \to \R$ be a $C^\infty$ function with compact support. There exists a compact set $\mathscr{K} \subset \MzD$ such that for any $\xi \in \MzD\setminus \mathscr{K}$, for all $n \geq 2$,
  \begin{equation*}
    \genMn \psi(\xi) = \bgenM\psi(\xi) = 0.
  \end{equation*}
    
  Thus, we restrict our attention to $\xi \in \mathscr{K}$ and define
  \begin{equation*}
    r^{x,y}_n(\xi) := \psi\left(\xi+\frac{\theta^{x,y}}{\sqrt{n}}\right)-\psi(\xi) - \frac{1}{\sqrt{n}}\langle \theta^{x,y}, \nabla\psi(\xi)\rangle - \frac{1}{2n}\langle \theta^{x,y}, \nabla^2\psi(\xi)\theta^{x,y}\rangle,
  \end{equation*}
  for $x,y \in \cD$. By the Taylor--Lagrange inequality,
  \begin{equation}\label{eq:rn}
    \sup_{n \geq 2} \sup_{\xi \in \mathscr{K}} \max_{x,y \in \cD} n^{3/2}|r^{x,y}_n(\xi)| < +\infty.
  \end{equation}
  On the other hand, for any $\xi \in \mathscr{K}$, we write
  \begin{equation*}
    \genMn \psi(\xi) = \sqrt{n} A^{(1)}_n(\xi) + A^{(2)}_n(\xi) + \frac{1}{\sqrt{n}}A^{(3)}_n(\xi),
  \end{equation*}
  where, letting 
  \begin{equation*}
    p^{\pi,n}_\cD(x,y) := p_\cD(x,y) + \frac{n}{n-1}q(x)\pi(y),
  \end{equation*}
  we write
  \begin{equation*}
    \begin{aligned}
      A^{(1)}_n(\xi) &:= \sum_{x,y \in \cD} \pi(x)p^{\pi,n}_\cD(x,y)\langle \theta^{x,y}, \nabla\psi(\xi)\rangle,\\
      A^{(2)}_n(\xi) &:= \sum_{x,y \in \cD} \left(\frac{n}{n-1}\pi(x)q(x)\xi(y) + \xi(x)p^{\pi,n}_\cD(x,y)\right)\langle \theta^{x,y}, \nabla\psi(\xi)\rangle\\
      & \quad + \frac{1}{2}\sum_{x,y \in \cD}\pi(x)p^{\pi,n}_\cD(x,y)\langle \theta^{x,y}, \nabla^2\psi(\xi)\theta^{x,y}\rangle,\\
      A^{(3)}_n(\xi) &:= \frac{1}{2} \sum_{x,y \in \cD} \left(\frac{n}{n-1}\pi(x)q(x)\xi(y) + \xi(x)p^{\pi,n}_\cD(x,y)\right)\langle \theta^{x,y}, \nabla^2\psi(\xi)\theta^{x,y}\rangle\\
      & \quad + n^{3/2}\sum_{x,y \in \cD} \left(\pi(x) + \frac{\xi(x)}{\sqrt{n}}\right)\left(p^{\pi,n}_\cD(x,y) + \frac{\sqrt{n}}{n-1}\pi(x)q(x)\xi(y)\right)r^{x,y}_n(\xi).
    \end{aligned}
  \end{equation*}
  
  A short computation shows that, for any $z \in \cD$,
  \begin{equation*}
    \sum_{x,y \in \cD} \pi(x)p^{\pi,n}_\cD(x,y) \theta^{x,y}(z) = \frac{1}{n-1}(\lambda\pi(z)-q(z)\pi(z)),
  \end{equation*}
  so that
  \begin{equation*}
    \lim_{n \to +\infty} \sup_{\xi \in \mathscr{K}} |\sqrt{n} A^{(1)}_n(\xi)| = 0.
  \end{equation*}
  Likewise, it follows from~\eqref{eq:rn} and the compactness of $\mathscr{K}$ that $|A^{(3)}_n(\xi)|$ is bounded uniformly in $n \geq 2$ and $\xi \in \mathscr{K}$, so that
  \begin{equation*}
    \lim_{n \to +\infty} \sup_{\xi \in \mathscr{K}} \left|\frac{1}{\sqrt{n}} A^{(3)}_n(\xi)\right| = 0.
  \end{equation*}
  
  Thus, it remains to show that
  \begin{equation*}
    \lim_{n \to +\infty} \sup_{\xi \in \mathscr{K}} \left|A^{(2)}_n(\xi) - \bgenM\psi(\xi)\right| = 0.
  \end{equation*}
  Using the convergence of $p^{\pi,n}_\cD(x,y)$ to $p^\pi_\cD(x,y)$, we may first observe that 
  \begin{equation*}
    \begin{aligned}
      & \lim_{n \to +\infty} \sum_{x,y \in \cD} \left(\frac{n}{n-1}\pi(x)q(x)\xi(y) + \xi(x)p^{\pi,n}_\cD(x,y)\right)\langle \theta^{x,y}, \nabla\psi(\xi)\rangle\\
      & \qquad = \sum_{x,y \in \cD} \left(\pi(x)q(x)\xi(y) + \xi(x)p^\pi_\cD(x,y)\right)\langle \theta^{x,y}, \nabla\psi(\xi)\rangle,
    \end{aligned} 
  \end{equation*}
  and that the limit is uniform in $\xi \in \mathscr{K}$. For any $z \in \cD$,
  \begin{equation*}
    \begin{aligned}
      & \sum_{x,y \in \cD} \left(\pi(x)q(x)\xi(y) + \xi(x)p^\pi_\cD(x,y)\right) \theta^{x,y}(z)\\
      &\qquad= \sum_{x \in \cD} \left(\pi(x)q(x)\xi(z) + \xi(x)p^\pi_\cD(x,z)\right) - \sum_{y \in \cD} \left(\pi(z)q(z)\xi(y) + \xi(z)p^\pi_\cD(z,y)\right)\\
      &\qquad= \lambda \xi(z) + (P^\pi_\cD)^*\xi(z) - \xi(z),
    \end{aligned}
  \end{equation*}
  where we have used Proposition~\ref{prop:QSD} to write $\langle \pi,q\rangle=\lambda$, the fact that $\xi \in \MzD$ to make the first sum over $y$ vanish, and the fact that $P^\pi_\cD$ is a stochastic matrix. By the Definition~\ref{defi:B0} of $B_0$, we therefore conclude that
  \begin{equation*}
    \sum_{x,y \in \cD} \left(\pi(x)q(x)\xi(y) + \xi(x)p^\pi_\cD(x,y)\right) \theta^{x,y} = B_0\xi,
  \end{equation*}
  so that the first term in the definition of $A^{(2)}_n(\xi)$ converges to $\langle B_0 \xi, \nabla \psi(\xi)\rangle$, uniformly in $\xi \in \mathscr{K}$. With similar arguments, it is immediate to show that the second term converges, uniformly in $\xi \in \mathscr{K}$, to $A^\pi_\cD :: \nabla^2\psi(\xi)$ defined by~\eqref{eq:A::}, which completes the proof.
\end{proof}

\subsection{Identification of \texorpdfstring{$\bgenM$}{M}}\label{ss:gauss:bgenM} Combining the results of Lemma~\ref{lem:ApiD} and Propositions~\ref{prop:spectral} and~\ref{prop:sympos}, we deduce that there exist $\zeta^1, \ldots, \zeta^{k-1} \in \MzD$ and $c^1, \ldots, c^{k-1} > 0$ such that 
\begin{equation*}
  \forall f \in \CzD, \qquad \langle A^\pi_\cD f, f\rangle = \sum_{l=1}^{k-1} c^l\langle \zeta^l, f\rangle^2,
\end{equation*}
where $k \geq 1$ is the cardinality of $\cD$ (see Appendix~\ref{app:linalg}).

Let $(\mathrm{w}^1_t)_{t \geq 0}, \ldots, (\mathrm{w}^{k-1}_t)_{t \geq 0}$ be independent Brownian motions, and let us consider the linear stochastic differential equation
\begin{equation*}
  \dd\bar{\upxi}_t = B_0\bar{\upxi}_t\dd t + \sum_{l=1}^{k-1} \sqrt{2c^l}\zeta^l \dd\mathrm{w}^l_t,
\end{equation*}
which defines a diffusion process $(\bar{\upxi}_t)_{t \geq 0}$ in $\MzD$. 

\begin{lemma}[Infinitesimal generator of $(\bar{\upxi}_t)_{t \geq 0}$]\label{lem:bgenM}
  The infinitesimal generator of $(\bar{\upxi}_t)_{t \geq 0}$ is the operator $\bgenM$ defined in Lemma~\ref{lem:cvMn}.
\end{lemma}
\begin{proof}
  Applying Ito's Formula to $\psi(\bar{\upxi}_t)$ for a smooth function $\psi : \MzD \to \R$, we observe that the claimed result reduces to checking the identity
  \begin{equation*}
    \forall \xi \in \MzD, \qquad \sum_{l=1}^{k-1} c^l\langle \zeta^l, \nabla^2 \psi(\xi)\zeta^l\rangle = A^\pi_\cD :: \nabla^2\psi(\xi).
  \end{equation*}
  For a fixed $\xi \in \MzD$, Proposition~\ref{prop:spectral} shows that there exist $g^1, \ldots, g^{k-1} \in \CzD$ and $d^1, \ldots, d^{k-1} \in \R$ such that
  \begin{equation*}
    \forall l=1, \ldots, k-1, \qquad \langle \zeta^l, \nabla^2 \psi(\xi)\zeta^l\rangle = \sum_{m=1}^{k-1} d^m \langle \zeta^l, g^m\rangle^2.
  \end{equation*}
  As a consequence, 
  \begin{equation*}
    \begin{aligned}
      \sum_{l=1}^{k-1} c^l\langle \zeta^l, \nabla^2 \psi(\xi)\zeta^l\rangle &= \sum_{l,m=1}^{k-1} c^ld^m\langle \zeta^l, g^m\rangle^2= \sum_{m=1}^{k-1} d^m \langle A^\pi_\cD g^m, g^m\rangle.
    \end{aligned}
  \end{equation*}
  Using the definition of the operator $A^\pi_\cD$ and of the measure $\theta^{x,y}$, we now write, for $m \in \{1, \ldots, k-1\}$,
  \begin{equation*}
    \begin{aligned}
      \langle A^\pi_\cD g^m, g^m\rangle = \mathcal{A}^\pi_\cD(g^m) &= \frac{1}{2} \sum_{x,y \in \cD} \pi(x)p^\pi_\cD(x,y)\left[g^m(y)-g^m(x)\right]^2\\
      &= \frac{1}{2} \sum_{x,y \in \cD} \pi(x)p^\pi_\cD(x,y)\langle \theta^{x,y}, g^m\rangle^2,
    \end{aligned}
  \end{equation*}
  so that, by the definition of $g^m$ and $d^m$,
  \begin{equation*}
    \begin{aligned}
      \sum_{m=1}^{k-1} d^m \langle A^\pi_\cD g^m, g^m\rangle &= \frac{1}{2} \sum_{x,y \in \cD} \pi(x)p^\pi_\cD(x,y)\sum_{m=1}^{k-1} d^m \langle \theta^{x,y}, g^m\rangle^2\\
      & =\frac{1}{2} \sum_{x,y \in \cD} \pi(x)p^\pi_\cD(x,y)\langle \theta^{x,y}, \nabla^2\psi(\xi)\theta^{x,y}\rangle = A^\pi_\cD :: \nabla^2\psi(\xi),
    \end{aligned}
  \end{equation*}
  where we have used~\eqref{eq:A::} to obtain the last identity.
\end{proof}

\begin{lemma}[Stationary distribution of $(\bar{\upxi}_t)_{t \geq 0}$]\label{lem:stat}
  The unique stationary distribution of $(\bar{\upxi}_t)_{t \geq 0}$ is the centered Gaussian measure on $\MzD$ with the covariance operator $K : \CzD \to \MzD$ defined by
  \begin{equation*}
    K := 2 \int_{s=0}^{+\infty} \ee^{sB_0} A^\pi_\cD \ee^{sB_0^*}\dd s.
  \end{equation*}
\end{lemma}
\begin{proof}
  Lemma~\ref{lem:ApiD} and Proposition~\ref{prop:sympos} show that the diffusion process $(\bar{\upxi}_t)_{t \geq 0}$ is uniformly elliptic on $\MzD$, which implies that transition semigroup of $(\bar{\upxi}_t)_{t \geq 0}$ has a positive density with respect to the Lebesgue measure on $\MzD$ (\cite[Corollary 1.1.6, p.~6]{AboFreIonJan12} and \cite[Eq.~(3.108), p.~82]{Pav14}). This `irreducibility' condition ensures uniqueness of stationary distributions (see for instance the arguments in~\cite[Section~3.1]{Pag01}). 
  
  On the other hand, \cite[Proposition~3.5, p.~80]{Pav14} shows that a Gaussian measure on $\MzD$ is stationary for $(\bar{\upxi}_t)_{t \geq 0}$ if and only if it is centered and its covariance operator $K : \CzD \to \MzD$ satisfies the Lyapunov equation
  \begin{equation*}
    B_0K + KB_0^* + 2A^\pi_\cD = 0,
  \end{equation*}
  which by Remark~\ref{rk:spB0} and Proposition~\ref{prop:Lyapgen} completes the proof.
\end{proof}

\begin{lemma}[Identification of the variance]\label{lem:idenVar}
  For any $f \in \CzD$,
  \begin{equation*}
    \begin{aligned}
      \langle Kf, f\rangle &= \Var_\pi(f) + 2\lambda \int_{s=0}^{+\infty} \ee^{2\lambda s} \Var_\pi(P^\pi_{s,\cD} f) \dd s\\
      & = \Var_\pi(f) + 2\lambda \int_{s=0}^{+\infty} \ee^{2\lambda s} \Var_\pi(Q_s f) \dd s.
    \end{aligned}
  \end{equation*}
\end{lemma}
\begin{proof}
  Let us fix $f \in \CzD$ and write
  \begin{equation*}
    \langle Kf,f\rangle = 2 \int_{s=0}^{+\infty} \langle A^\pi_\cD \ee^{sB_0^*}f, \ee^{sB_0^*}f\rangle\dd s.
  \end{equation*}
  For any $s \geq 0$, it follows from Definition~\ref{defi:B0} and~\eqref{eq:PtD} that 
  \begin{equation*}
    \ee^{sB_0^*}f = \ee^{s(P^\pi_\cD-(1-\lambda)I_{\CzD})}f = \ee^{\lambda s} P^\pi_{s,\cD} f,
  \end{equation*}
  whence
  \begin{equation*}
    \langle Kf,f\rangle = 2 \int_{s=0}^{+\infty} \ee^{2\lambda s} \mathcal{A}^\pi_\cD(P^\pi_{s,\cD} f)\dd s.
  \end{equation*}
  On the other hand, it follows from the definition of $\mathcal{A}^\pi_\cD$ that for any $s \geq 0$,
  \begin{equation*}
    \begin{aligned}
      \mathcal{A}^\pi_\cD(P^\pi_{s,\cD} f) &= -\sum_{x \in \cD} \pi(x) P^\pi_{s,\cD} f(x) \frac{\dd}{\dd s} P^\pi_{s,\cD} f(x)\\
      & = -\frac{1}{2} \frac{\dd}{\dd s} \sum_{x \in \cD} \pi(x) \left(P^\pi_{s,\cD} f(x)\right)^2 = -\frac{1}{2} \frac{\dd}{\dd s} \Var_\pi\left(P^\pi_{s,\cD} f\right),
    \end{aligned}
  \end{equation*}
  where in the last equality we used the fact that $\langle \pi,P^\pi_{s,\cD}f\rangle = \langle \pi,f\rangle=0$. As a consequence,
  \begin{equation*}
    \langle Kf,f\rangle = - \int_{s=0}^{+\infty} \ee^{2\lambda s} \frac{\dd}{\dd s} \Var_\pi\left(P^\pi_{s,\cD} f\right)\dd s,
  \end{equation*}
  and using Lemma~\ref{lem:spectpi}~\eqref{it:spectpi:2} to integrate the right-hand side by parts leads to the first claimed expression. The second expression follows from the fact that by Lemma~\ref{lem:spectpi}~\eqref{it:spectpi:3}, the operators $P^\pi_{s,\cD}$ and $Q_s$ coincide on $\CzD$.
\end{proof}

\subsection{Proof of Theorem~\ref{theo:CLT}}\label{ss:gauss:pf} We may now complete the proof of Theorem~\ref{theo:CLT}.

\begin{proof}[Proof of Theorem~\ref{theo:CLT}]
  We recall from~\eqref{eq:upxininf} that the fluctuation field $\sqrt{n}(\upeta^n_\infty-\pi)$ is denoted by $\upxi^n_\infty$ .
  
  By Lemma~\ref{lem:tight}, the sequence $(\upxi^n_\infty)_{n \geq 2}$ is tight in $\MzD$. Therefore by Prohorov's Theorem~(\cite[Theorem~5.1, p.~59]{Bil99}), any subsequence possesses a further subsequence, which we shall index by $n_m$, which converges in distribution to a random variable $\bar{\upxi}_\infty$ in $\MzD$. In particular, for any $C^\infty$ function with compact support $\psi : \MzD \to \R$,
  \begin{equation*}
    0 = \Exp\left[\genMnm\psi(\upxi^{n_m}_\infty)\right] = \Exp\left[(\genMnm-\bgenM)\psi(\upxi^{n_m}_\infty)\right] + \Exp\left[\bgenM\psi(\upxi^{n_m}_\infty)\right]\to \Exp\left[\bgenM\psi(\bar{\upxi}_\infty)\right],
  \end{equation*} 
  where we have used Lemma~\ref{lem:cvMn} to get that $\Exp[(\genMnm-\bgenM)\psi(\upxi^{n_m}_\infty)]$ converges to $0$. 
  
  Thus, we deduce from Lemma~\ref{lem:bgenM} that the law of $\bar{\upxi}_\infty$ is a stationary distribution for $(\bar{\upxi}_t)_{t \geq 0}$, which by Lemma~\ref{lem:stat} entails its identification and yields the convergence of the whole sequence $(\upxi^n_\infty)_{n \geq 2}$ to $\bar{\upxi}_\infty$, with the asymptotic covariance given by Lemma~\ref{lem:idenVar}.
\end{proof}

\appendix
\section{Complements on operators and Lyapunov equations}\label{app:linalg}

In this appendix, we recall some elementary results of linear algebra which are useful in our framework. We denote by $k \geq 1$ the cardinality of $\cD$, so that $\MzD$ and $\CzD$ are linear spaces of dimension $k-1$. 

\subsection{Diagonalisation of symmetric operators} Following Remark~\ref{rk:idendual}, recall that the spaces $\MzD$ and $\CzD$ are identified with each other's dual. In this framework, the operators $N : \CzD \to \MzD$ and $R : \MzD \to \CzD$ are called symmetric if they coincide with their adjoint operators $N^* : \CzD \to \MzD$ and $R^* : \MzD \to \CzD$.

\begin{proposition}[Diagonal form of symmetric operators]\label{prop:spectral}
  Let $N : \CzD \to \MzD$ and $R : \MzD \to \CzD$ be symmetric operators.
  \begin{enumerate}[label=(\roman*),ref=\roman*]
    \item There exist a basis $(\zeta^1, \ldots, \zeta^{k-1})$ of $\MzD$ and $c^1, \ldots, c^{k-1} \in \R$ such that, for any $f \in \CzD$, $\langle N f, f\rangle = \sum_{l=1}^{k-1} c^l \langle \zeta^l, f\rangle^2$.
    \item There exist a basis $(g^1, \ldots, g^{k-1})$ of $\CzD$ and $d^1, \ldots, d^{k-1} \in \R$ such that, for any $\xi \in \MzD$, $\langle \xi, R\xi\rangle = \sum_{l=1}^{k-1} d^l \langle \xi, g^l\rangle^2$.
  \end{enumerate}
\end{proposition}

\subsection{Positive definite operators} A symmetric operator $N : \CzD \to \MzD$ is called \emph{positive definite} if it satisfies $\langle Nf, f\rangle > 0$ for any $f \in \CzD \setminus \{0\}$. A similar definition holds for operators $R : \MzD \to \CzD$.  

\begin{proposition}[On symmetric and positive definite operators]\label{prop:sympos}
  Let $N : \CzD \to \MzD$ be a symmetric operator. The following conditions are equivalent.
  \begin{enumerate}[label=(\roman*),ref=\roman*]
    \item $N$ is positive definite.
    \item The numbers $c^1, \ldots, c^{k-1}$ provided by Proposition~\ref{prop:spectral} are positive.
    \item For any norm $\|\cdot\|_{\CzD}$ on $\CzD$, there exists $c_N \in (0,+\infty)$ such that for all $f \in \CzD$, $\langle Nf, f\rangle \geq c_N \|f\|^2_{\CzD}$.
  \end{enumerate}
  Besides, under any of these conditions, $N$ is invertible and its inverse $N^{-1} : \MzD \to \CzD$ is a symmetric and positive definite operator.
\end{proposition}
With obvious adjustments, a similar statement holds for symmetric and positive definite operators $R: \MzD \to \CzD$.

\subsection{Lyapunov equation} In this section, we recall a standard result regarding Lyapunov equations, see for instance~\cite[Theorem~1.1.7, p.~6]{AboFreIonJan12}.
\begin{proposition}[Solution to Lyapunov equations]\label{prop:Lyapgen}
  Let $\tilde{B}_0 : \MzD \to \MzD$ be an operator of which all the eigenvalues have a negative real part, and let $A : \CzD \to \MzD$ be a symmetric operator.
  \begin{enumerate}[label=(\roman*),ref=\roman*]
    \item The integral
    \begin{equation*}
      \tilde{K} := 2 \int_{s=0}^{+\infty} \ee^{s\tilde{B}_0} A \ee^{s\tilde{B}_0^*}\dd s
    \end{equation*}
    is finite and defines a symmetric operator $\CzD \to \MzD$.
    \item $\tilde{K}$ is the unique symmetric solution to the Lyapunov equation
    \begin{equation*}
      \tilde{B}_0 \tilde{K} + \tilde{K} \tilde{B}_0^* + 2A = 0.
    \end{equation*}
    \item If $A$ is positive definite, then so is $\tilde{K}$.
  \end{enumerate}
\end{proposition}

\section*{Acknowledgements}
We thank Arnaud Guyader, Laurent Miclo and Mathias Rousset for fruitful discussions.

\bibliographystyle{alea3}
\bibliography{clt}

\end{document}